\documentclass{amsart}

\usepackage{nicefrac}
\usepackage{slashbox}
\usepackage{float}

\usepackage{amssymb}
\usepackage{graphicx} 
\usepackage{tikz}
\tikzset{>=latex}
\usetikzlibrary{arrows,backgrounds,decorations,automata}

\usepackage{subfig}

\newtheorem{theorem}{Theorem}[section]
\newtheorem{lemma}[theorem]{Lemma}

\newtheorem{prop}[theorem]{Proposition}
\newtheorem{conj}{Conjecture}

\theoremstyle{definition}

\theoremstyle{remark}
\newtheorem{remark}[theorem]{Remark}

\newcommand{\diag}{I}

\numberwithin{equation}{section}

\newcommand{\Ref}[1]{(\ref{#1})}

\newcommand{\abar}{a^{-1}}
\newcommand{\bbar}{b^{-1}}
\newcommand{\qbar}{q^{-1}}
\newcommand{\suba}[1]{\langle a^{#1} \rangle}

\newcommand{\Triv}{Triv}
\newcommand{\cogrowthF}{MR3043436}

\newcommand{\LS}{MR0577064}

\newcommand{\Cohen}{Cohen}
\newcommand{\DykemaBS}{DykemaBS}
\newcommand{\Flajolet}{Flajolet}

\newcommand{\Grig}{Grig}
\newcommand{\KouksovFreeProd}{MR1689726}

\newcommand{\KouksovRational}{MR1487319}

\newcommand{\GStald}{MR2415304}

\newcommand{\Mann}{MR2894945}
\newcommand{\Wagon}{MR1251963}

\newcommand{\LipsDiag}{MR929767}
\newcommand{\Stanley}{stanley1980}
\newcommand{\StanleyBookTwo}{MR1676282}
\newcommand{\ZW}{MR808671}
\newcommand{\KW}{MR689038}

\newcommand{\Kauers}{kauers}
\newcommand{\KauersGuess}{kauersguess}

\newcommand{\Hammersley}{hammersley1982}

\newcommand{\Woess}{MR731608}

\newcommand{\HumphriesA}{MR1466893}
\newcommand{\HumphriesB}{MR1426519}
\newcommand{\Poole}{MR0111886}
    
\begin{document}

\title{The cogrowth series for  $\mathrm{BS}(N,N)$ is D-finite}

\author[M. Elder]{M. Elder}
\address{School of Mathematical \& Physical Sciences, The~University~of~Newcastle, Callaghan, New South Wales, Australia}
\email{murray.elder@newcastle.edu.au}
\author[A. Rechnitzer]{A. Rechnitzer}
\address{Department of Mathematics, University of British Columbia, Vancouver, British Columbia, Canada}
\email{andrewr@math.ubc.ca}
% %
\author[E.J.~Janse~van~Rensburg]{E.J.~Janse~van~Rensburg}
\address{York University, Toronto, Ontario, Canada}
\email{rensburg@mathstat.yorku.ca}
\author[T. Wong]{T. Wong}
\address{Department of Mathematics, University of British Columbia, Vancouver, British Columbia, Canada}
\email{twong@math.ubc.ca}



\date{Semptember 2013}

\begin{abstract}
We compute the cogrowth series for
Baumslag-Solitar groups $\mathrm{BS}(N,N) = \langle a,b | a^N b = b a^N \rangle$, which we show to be D-finite. It follows that 
 their cogrowth rates are algebraic numbers.

\end{abstract}

\maketitle

\section{Introduction}

The function $c:\mathbb N\rightarrow \mathbb N$ where $c(n)$ is the number of
words of length $n$  in the generators and inverses of generators of a finitely generated group that
represent the identity element is called the {\em cogrowth function} and the
corresponding generating function is called the \emph{cogrowth series}. The
rate of exponential growth of the cogrowth function  $\limsup c(n)^{1/n}$ is the {\em cogrowth}
of the group (with respect to a chosen finite  generating set).  Note that cogrowth can also be defined by counting only {\em freely reduced} trivial words --- see Remark \ref{rem:freered} at the end of this section.

In this article we study the cogrowth of the groups $\mathrm {BS}(N,M)$ with presentation
\[ \mathrm{BS}(N,M) = \langle a,b  \ | \  a^Nb= b a^M  \rangle \] 
for positive integers $N,M$.  We prove in Theorem \ref{thm:main} that for groups $\mathrm{BS}(N,N)$ the cogrowth series is D-finite, that is, satisfies a linear differential equation with polynomial coefficients.

%Introduce algebraic, P-recursive and D-finite.

The class of D-finite  (or holonomic) functions includes rational and algebraic functions, 
 and many of the most famous functions in mathematics and physics. See 
\cite{\StanleyBookTwo, \Stanley} for background on D-finite generating 
functions. If $\{a_n\}$ is a sequence and $A(z)=\sum_n a_nz^n$ is its 
corresponding generating function then $A(z)$ is D-finite if and only if 
$\{a_n\}$ is P-recursive (satisfies a linear recurrence with polynomial 
coefficients). D-finite functions are closed under addition and multiplication, 
and  the composition of D-finite function with an algebraic function is 
D-finite \cite{\Stanley}. Further, if a generating function is D-finite and the 
differential equation is known, then the coefficients of corresponding sequence 
can be computed quickly and their asymptotics are readily computed (see for 
example \cite{\ZW}).

The class of group presentations for which the cogrowth series has been computed explicitly (in terms of a closed-form expression, or system of simple recurrences, for example) is limited. Kouksov proved that the cogrowth series is a rational function if and only if the group is finite \cite{\KouksovRational}, and  computed closed-form expressions for some free products of finite groups and free groups \cite{\KouksovFreeProd} which are algebraic functions.  Humphries gave recursions and closed-form functions  for various abelian groups \cite{\HumphriesA,\HumphriesB}. %So all known examples are D-finite. 

We note that Dykema and Redelmeier have also tried to compute cogrowth for general Baumslag-Solitar groups \cite{\DykemaBS}, and the problem appears to be a difficult one.

 Grigorchuk and
independently Cohen \cite{\Cohen, \Grig} proved that a finitely generated group
is  amenable if and only if its cogrowth rate is twice the number of generators. %(recall we define cogrowth as counting all trivial words, not just freely reduced words).
For more background on amenability and  cogrowth see \cite{\Mann, \Wagon}.
The free group on two (or more) letters is
known to be non-amenable, and  subgroups of amenable groups are also amenable.
It follows that if a group contains a subgroup isomorphic to the free
group on two generators, then it cannot be amenable.  $\mathbb{Z}^2 \cong \mathrm{BS}(1,1)$ is amenable, while 
for $N>1$ the subgroup of 
 $\mathrm{BS}(N,N)$ generated by $at$ and $at^{-1}$ is free, so these groups are non-amenable. 
 We compute  cogrowth rates for $\mathrm{BS}(N,N)$ for $N\leq 10$ (see  Table \ref{tab_mu_NN}), and observe that the rate appears to converge to that of the free group of rank 2.  This is in line with the result of Guyot and Stalder  \cite{\GStald} that the  the limit (in the space of   {\em marked groups}) of  $\mathrm{BS}(N,M)$ as $N,M\rightarrow\infty$ is the free group of rank 2.

In \cite{\cogrowthF} a numerical method was used to find bounds for the cogrowth 
of groups, and a lower bound  %(3.78522)  
 for the cogrowth rate of 
$BS(N,M)$ for small values of $N,M$ were computed. A significant improvement of this numerical work has been 
undertaken by the authors, which can be found in \cite{\Triv}. 
%%%actual value computed here in section -- is 3.792765039

The article is organised as follows. In Section \ref{sec:ZZ}
we briefly explain the well-known result that the cogrowth for $\mathrm{BS}(1,1)\cong \mathbb Z^2$ is D-finite and not algebraic.
 In Section \ref{sec:exact} we introduce some two variable generating functions that count various words in Baumslag-Solitar groups $\mathrm{BS}(N,M)$, and give a system of equations that they satisfy in Proposition \ref{prop:GLKeqns}.
% In Section \ref{sec:exact} we establish some relations between various generating functions for cogrowth in general Baumslag-Solitar groups. 
In Section \ref{sec:NN} we restrict  these equations to the case when $N=M$,  
and prove that the cogrowth series for $\mathrm{BS}(N,N)$ is D-finite. In 
Section \ref{sec:rates} we
compute precise numerical values for the cogrowth rates, and in Section \ref{sec:numerical} we discuss computational work to find explicit formulae for the differential equations proved above, and conjecture an
asymptotic form for the cogrowth series.

\begin{remark}\label{rem:freered}
We note that the cogrowth of a group is often defined in terms of freely 
reduced words. Let $d:\mathbb N\rightarrow \mathbb N$ where $d(n)$ is the 
number of freely reduced words of length $n$. The associated generating 
functions of $c(n)$ and $d(n)$ are related via the following algebraic 
substitutions.
%\textbf{in the refs the formula is for general $k$, here $k=2$}
%\begin{lemma}[Lemma 1 of \cite{\Woess}; \cite{\KouksovRational}]
%  Let $C(z) = \sum c(n) z^n$ and $D(z)=\sum d(n) z^n$ be the generating 
%functions associated to $c(n)$ and $d(n)$, and assume the group in question is generated by two elements. 
%Then
%\begin{align}
%D(z) &= \frac{1-z^2}{1+3z^2} C\left(\frac{z}{1+3z^2} \right) \\
%C(z) &= \frac{-1+2\sqrt{1-12z^2}}{1-16z^2} D\left( \frac{1-\sqrt{1-12z^2}}{6z} 
%\right)
%\end{align}
%\end{lemma}
\begin{lemma}[Lemma 1 of \cite{\Woess}; \cite{\KouksovRational}]
\label{lemma woe}
  Let $C(z) = \sum c(n) z^n$ and $D(z)=\sum d(n) z^n$ be the generating 
functions associated to $c(n)$ and $d(n)$, and assume the group in question is generated by $p$ elements and their inverses. 
Then
\begin{align}
D(z) &= \frac{1-z^2}{1+(2p-1)z^2} C\left(\frac{z}{1+(2p-1)z^2} \right) \\
C(z) &= \frac{1-p+p\sqrt{1-4(2p-1)z^2}}{1-4p^2z^2} D\left( \frac{1-\sqrt{1-4(2p-1)z^2}}{2(2p-1)z} 
\right)
\end{align}
\end{lemma}
In this work we show that $C(z)$ is D-finite for $BS(N,N)$. The closure 
properties of D-finite functions then imply that $D(z)$ is also D-finite. More 
precisely, D-finite functions are closed under multiplication and the result of 
composing a D-finite function with an algebraic function is also D-finite (see 
Theorems~2.3 and~2.7 in~\cite{\Stanley}). 
\end{remark}

\section{$\mathrm{BS}(1,1)$}
\label{sec:ZZ}
The contents of this section are well known, see for example sequence~A002894 
in~\cite{Sloane} or page~90 of \cite{\Flajolet}. The group $ 
\mathrm{BS}(1,1)=\langle a, b \ | \ ab=ba\rangle$ is simply the free abelian 
group $\mathbb Z^2$ of rank~2. The {\em Cayley graph} is just the square grid, 
and  trivial words correspond to  closed  paths of even length starting and 
ending at the origin.

Now rotate the grid $45^\circ$ and rescale by $\sqrt 2$ --- see Figure \ref{fig:ZZ}.
Each step in a 
closed path changes the $x$-ordinate by $\pm1$. At the same time, each step 
changes the $y$-ordinate by $\pm1$, and these two processes are independent. In 
a path of $2n$-steps, $n$ steps must increase the $x$-ordinate and $n$ must 
decrease it and so giving $\binom{2n}{n}$ possibilities. The same occurs 
independently for the $y$-ordinates and so we get another 
factor of $\binom{2n}{n}$. Hence the total number of  possible trivial words of length $2n$ is 
$\binom{2n}{n}^2$.

\begin{center}
\begin{figure}[h]
\begin{tikzpicture}[scale=.8]

%   \node (a) at (.7,4.3) {id};

   \node (a) at (1,4) {$\bullet$};

   \node (aa) at (1.2,4.8) {$a$};

      \draw[->] (0,3)--(1.6,4.6);

      \draw (1,4)--(3,6);
      \draw (3,6)--(4,7);
  
      \draw  (3,6)--(2,7);
      \draw (2,7)--(3,6);

   \node (bb) at (4.8,6.8) {$b$};

    \draw[->] (4,7)--(4.6,6.4);
    \draw (4,7)--(5,6);

      \draw (5,6)--(4,5);
  \draw (4,5)--(5,6);
      \draw  (5,6)--(6,5);
      \draw (6,5)--(5,4);
      \draw (5,4)--(6,3);
      \draw (6,3)--(5,2);
      \draw  (5,2)--(4,1);
      \draw (4,1)--(3,2);
      \draw (3,2)--(2,1);
      \draw (2,1)--(1,2);
      \draw (1,2)--(0,3);

  \draw[dotted] (1,4)--(3,2);
  \draw[dotted] (2,5)--(5,2);
  \draw[dotted] (3,6)--(5,4);

  \draw[dotted] (1,2)--(4,5);
  \draw[dotted] (3,2)--(5,4);
      \end{tikzpicture}
 \caption{A trivial word $aa\bbar bab\abar ab\abar b \abar\abar\bbar\abar\bbar\bbar a$
in the Cayley graph of $\langle a, b \ | \ ab=ba\rangle$.}% We depict the Cayley graph rotated 45 degrees to make counting trivial words easy.}
%%% there is no proper way to orient a CG, so its a bit weird to say that

   \label{fig:ZZ}
\end{figure}
\end{center}

Let $c_{2n}=\binom{2n}{n}^2$ and $c_{2n+1}=0$, and notice that $(n+1)^2c_{2n+2}=4(2n+1)^2c_{2n}$.  Then the sequence 
 $\{c_n\}$ satisfies the polynomial recurrence  %\footnote{In the entry for A002894 \cite{Sloane} this recurrence is given incorrectly.}
\[ (\nicefrac{n}{2}+1)^2c_{n+2}=4(n+1)^2c_n\] and is therefore {\em P-recursive}, which implies that the corresponding cogrowth series  satisfies a linear differential equation with
polynomial coefficients (Theorem 1.5  \cite{\Stanley}), that is, the cogrowth series is  {\em D-finite}.

One can  show that the cogrowth series is not algebraic by considering its
asymptotics. In particular, the coefficients of an algebraic function must grow 
as $A \mu^n n^\gamma$ where $\mu$ is an algebraic number, and $\gamma$ 
belongs to the set $\mathbb{Q}\setminus\{-1,-2,-3,\dots\}$ (see Theorem~D from 
\cite{flajolet1987}). An application of Stirling's formula shows that 
\begin{align}
c_n 
&\sim 4^n \cdot \frac{2}{\pi n},
\end{align}
and so the factor of $n^{-1}$ implies the corresponding generating function is 
not algebraic.

\section{Series for general Baumslag-Solitar groups}
\label{sec:exact}
Let us fix the following notation. If words $u,v$ are identical as strings we 
write $u\equiv v$, and if they represent the same group element we write $u=v$.
If a word $w$ represents an element in a subgroup $U$, we write $w \in U$.

Consider the  Baumslag-Solitar group \[ \mathrm{BS}(N,M) = \langle a,b  \ | \ a^Nb= b a^M  \rangle \] 
with $N,M$ positive integers.

Any word in $\{a^{\pm1}, b^{\pm1}\}^*$ can be transformed into a {\em normal form}  (\cite{\LS}  p.181) for
the corresponding group element by ``pushing'' each $a$ and $\abar$ in the
word as far to the right as possible using the identities 
\begin{align*}
a^{\pm 1}a^{\mp 1}&=1,  & b^{\pm 1}b^{\mp 1}&=1,\\
a^{\pm N} b &= b a^{\pm M} &  a^{\pm M} \bbar &= \bbar a^{\pm N},\\ 
a^{-i}b&=a^{N-i}ba^{-M} &  a^{-j}\bbar&= a^{M-j} \bbar a^{-N} 
\end{align*}
where $0<i<N $ and $ 0<j<M$, so that only positive powers of $a$ appear before a 
$b^{\pm 1}$ letter.
 The resulting normal form can be written as $P a^k$, where  $P$ is a freely reduced word in the alphabet $\{b, ab, \dots a^{N-1}b,
\bbar, a\bbar, \dots a^{M-1}\bbar \}$. Call $P$ the {\em prefix} and $k$ the {\em $a$-exponent} of the normal form word $Pa^k$.

Recall 
 Britton's lemma (\cite{\LS} p.181)  which in the case of $\mathrm{BS}(N,M)$ states that if  a trivial word $w\in\{a^{\pm1}, b^{\pm1}\}^*$ is freely reduced  and contains a $b^{\pm 1}$ letter, then $w$ must contain a subword of the form $ba^{lN}\bbar$ or $\bbar a^{lM} b$ for some $l\in\mathbb Z$.

Let $\mathcal{H}$ be the set of words  in  $\{a^{\pm1}, b^{\pm1}\}^*$ that 
represent elements in  $\suba{}$. Define $g_{n,k}$ to be the number of words 
$w\in \mathcal{H}$ of length $n$ having normal form with $a$-exponent $k$.

The cogrowth function for  $\mathrm{BS}(N,M)$ can be obtained from $g_{n,k}$ by setting $k=0$, as the next lemma shows. The reason for considering a more general function is that  we found  the corresponding two-variable generating function  much easier to work with than the cogrowth series directly.

\begin{lemma}\label{lem:zeroterms}  $g_{n,0}=c(n)$ where $c:\mathbb N\rightarrow \mathbb N$ is the cogrowth function for $\mathrm{BS}(N,M)$.  \end{lemma}
\begin{proof}
$g_{n,0}$ counts words $u\in \suba{}$ with normal form $Pa^0$, so $Pa^0u^{-1}=Pa^k=1$ for some $k$ so by Britton's Lemma 
$P$ must be the empty string.\end{proof}

\begin{lemma} \label{lem:gnkneg}
$g_{n,k} = g_{n,-k}$.\end{lemma}
\begin{proof}
Exchange
each $a$ by $\abar$ and vice versa in all words of length $n$ equal to $a^k$ to obtain all words of length $n$ equal to $a^{-k}$.
\end{proof}

Define two subsets of $\mathcal{H}$ as follows.
\begin{itemize}
 \item Let $\mathcal{L}$ be the set of words  $w\in\mathcal{H}$ such that for any prefix $u$ of $w$, $u$ does not have normal form $\bbar
a^j$ for any integer $j$.

 \item Let $\mathcal{K}$ be the set of  words $w\in\mathcal{H}$ such that for any prefix $u$ of $w$, $u$ does not have normal form  $ba^j$ for any integer 
$j$.
\end{itemize}
Define $l_{n,j}, k_{n,j}$ to be the number of words of length $n$ and 
$a$-exponent $j$ in $\mathcal L, \mathcal K$ respectively. 

\begin{lemma}
\label{lem Lchar}
Let $w\in\mathcal H$. Then 
 $w\in \mathcal L$ if and only if $w\not\equiv 
x\bbar y$ with $x\in\suba{N}$. Similarly $w\in \mathcal K$ if and only if 
 $w\not\equiv xby$ with $x\in\suba{M}$.
\end{lemma}
\begin{proof}
If $w\equiv 
x\bbar y$ with $x\in\suba{N}$, then $w=a^{iN}\bbar y=\bbar a^{iM}y$ so $w\not\in\mathcal L$.
Conversely if $w\not\in\mathcal L$ then $w\equiv uv$ with $u=\bbar
a^j$ for some $j$. Let $Pa^k$ be the normal form for $v$.
Since $w\in \mathcal H$ then $w=a^i$ for some $i$, and so $$1=wa^{-i}=uva^{-i}=\bbar a^j Pa^{k-i}.$$  Since $P$ contains no subwords of the form 
$ba^{lN}\bbar$ or $\bbar a^{lM} b$, 
Britton's lemma  implies that $j=lM$ and $P$ starts with a $b$, so $w\equiv xy$ with $x=\bbar a^{lM} b$.
The result for $\mathcal K$ follows by a similar argument.
\end{proof}

Define the following two-variable generating functions
\begin{align*}
  G(z;q) &= \sum_{n,j} g_{n,j} z^n q^j, \\
  L(z;q) &= \sum_{n,j} l_{n,j} z^n q^j, &
  K(z;q) &= \sum_{n,j} g_{n,j} z^n q^j.
\end{align*}
These are all formal power series in $z$ with coefficients that are Laurent 
polynomials in $q$. The functions $L$ and $K$ are related as follows.
\begin{lemma}
In any group $\mathrm{BS}(N,M)$ we have
 $L(z;1)=K(z;1)$.\end{lemma}
\begin{proof}
$L(z;1)=\sum_n\left(\sum_k l_{n,k}\right)z^n$, where the inner sum is the number of  words of length $n$ in $\mathcal L$. 
If $w\in\mathcal H\setminus \mathcal L$ then $w\equiv xb^{-1}y$ with $x\in\suba{N}$. Assume $x$ is chosen to be  of minimal length -- that is, 
if $x \equiv ub^{-1}v$ then $u$ is not in $\suba{N}$.
Since $w\in\mathcal H$ then $x\bbar y w^{-1}=a^{Ni}\bbar y a^i=1$ so by Britton's lemma we must have $y=zbt$ with $z\in\suba{M}$. Again choose $z$ to be minimal. 
Then  $zbxb^{-1}t\in\mathcal H\setminus \mathcal K$ has the same length as $w$. Since $x,z$ are uniquely determined,  this gives  a bijection between $\mathcal H\setminus \mathcal L$ and $\mathcal H\setminus \mathcal K$, and therefore between $\mathcal L$ and $\mathcal K$. 
\end{proof}

In the case that $N=M$ we can say even more.
\begin{lemma}\label{lem:LequalsK}
In  $\mathrm{BS}(N,N)$ we have $L(z;q)=K(z;q)$.\end{lemma}\begin{proof}
If $w\in \mathcal L$ has length $n$ then replacing $b$ by $\bbar$ and vice versa, we obtain a word which has a prefix equal to $ba^j$ and of length $n$, and   represents the same power of $a$ as $w$, so is in  $\mathcal K$. 
\end{proof}
Note that if $N\neq M$ then the above proof breaks down --- replacing $b$ by 
$\bbar$ in $u=ba^N\bbar\in\mathcal L$ gives a word not in $\suba{}$, so the 
resulting word is not in $\mathcal K$.

Next for $d,e\in\mathbb{N}$ define an operator $\Phi_{d,e}$ which acts on 
Laurent polynomials in $q$ by
\begin{align*}
 \Phi_{d,e}\circ \left(Aq^k\right) &=
\begin{cases}
Aq^{ej} & \mathrm{if} \ k=dj\\
0 & \mathrm{otherwise}.
\end{cases}
\end{align*}
For example, $\Phi_{2,3}(4zq+5zq^4)= 5zq^6$. We then extend this operator to 
power series in $z$ with coefficients that are Laurent polynomials in $q$, in 
the obvious way
\begin{align*}
\Phi_{d,e} \circ \left(\sum_n z^n \sum_{k} c_{n,k} q^k \right)
  &= \sum_n z^n \sum_j c_{n,dj} q^{ej}.
\end{align*}
Note that  $\Phi_{N,N}$ simply deletes all terms except those of the form $cz^iq^{jN}$.

With these definitions we can write down a set of equations satisfied by the
functions $G, L$ and $K$.
\begin{prop}
\label{prop:GLKeqns}
The generating functions $G \equiv G(z;q), L \equiv L(z;q)$ and $K \equiv K(z;q)$ satisfy the following system of equations.
\begin{align*}
 L &=  1 + z(q+\qbar)L + z^2 L\cdot
  \left[  \Phi_{N,M}\circ L + \Phi_{M,N} \circ K \right] 
   - z^2 \left[ \Phi_{M,N} \circ K \right] \cdot
  \left[ \Phi_{N,N} \circ L \right], \\
  K & = 1 + z(q+\qbar)K + z^2 K \cdot
  \left[ \Phi_{M,N} \circ K + \Phi_{N,M} \circ L \right]
  - z^2 \left[ \Phi_{N,M} \circ L \right] \cdot 
  \left[ \Phi_{M,M} \circ K \right],
\intertext{and}
  G &= 1 + z(q+\qbar)G 
  + z^2 \left[ \Phi_{N,M} \circ L +  \Phi_{M,N} \circ K \right]G.
\end{align*}
\end{prop}

\begin{proof}

We first establish the equation 
\begin{align*}  
  G 
% &= 1 + z(q+\qbar)G   + z^2  \left[ \Phi_{N,M} \circ L +  \Phi_{M,N} \circ K 
% \right]G. \\
  &= 1 + zqG + z\qbar G + z^2 \left[ \Phi_{N,M} \circ L \right] G + 
  z^2 \left[ \Phi_{M,N} \circ K \right]G
\end{align*}
Factor words in $\mathcal{H}$ recursively by
considering the first letter in any word $w \in \mathcal{H}$.
%(see
%Figure~\ref{fig:decompG}). 
%\begin{figure}[h]
% \centering
%\includegraphics[height=3cm]{D1.pdf}
%\caption{DOTS LOOKS LIKE BULLET POINTS, ADJUST PIC. Any word in $\mathcal{H}$ can be decomposed by considering its first
%letter. There are 5 possibilities: $w$ is empty, or starts with $a,\abar,b,\bbar$. The subwords $g, g' \in \mathcal{H}$, $L\in\mathcal{L}$ and
%$K\in\mathcal{K}$.}
%\label{fig:decompG}
%\end{figure}
This gives five cases which will correspond to the five terms on the RHS of the 
above equation:
\begin{itemize}
\item $w$ is the empty word. This is counted by the 1 in the expression for $G$.
\item The first letter is $a$. Then $w \equiv a v$  for some $v \in 
\mathcal{H}$. If $w$ has length $n$ and $a$-exponent $k$, then $v$ has 
length $n-1$ and $a$-exponent $k-1$. Summing over the contribution of all 
possible such $w$ gives $zq G(z;q)$ at the level of generating functions.

%  since the number of words counted by $g_{n,k}$ of this form 
% is the number of words counted in $\mathcal H$ of length $n-1$ and $a$-exponent 
% $k-1$.
\item The first letter is $\abar$. Then $w \equiv \abar
v$ for some $v \in \mathcal{H}$.
% increasing the length by $1$ and altering the $a$-exponent by $\pm1$. 
At the level of generating functions this gives
$z\qbar G(z;q)$, since the number of words counted by $g_{n,k}$ of this form is the number of words counted in $\mathcal H$ of length $n-1$ and $a$-exponent $k+1$.

\item The first letter is $b$. Write  $w \equiv u v$ where $u$ is the shortest
prefix of $w$ so that $u \in \suba{}$. Thus, $u \equiv b u' \bbar$ for some $u' \in
\suba{N}$ and so $u \in \suba{M}$. The minimality of $u$ ensures $u' \in 
\mathcal{L}$.

It follows that words $w\equiv bu'b^{-1}v$ of length $n$ and $a$-exponent $k$ are counted by
$$\sum_{i=0}^{n-2}\sum_j \left(l_{i,jN}\right)\left( g_{n-i-2, k-jM}\right)$$
where $l_{i,jN}$ counts the words $u'$ of length $i$ and $a$-exponent $jN$,
 and $g_{n-i-2, k-jM}$ counts the words $v$ of length $n-i-2$ and $a$-exponent $k-jM$.

So the term $$\left(\sum_{i=0}^{n-2}\sum_j \left(l_{i,jN}\right)\left( g_{n-i-2, k-jM}\right)\right)z^nq^k = z^2\left(\sum_{i=0}^{n-2}\sum_j \left(l_{i,jN}\right)\left( g_{n-i-2, k-jM}\right)\right)z^{n-2}q^k$$
is the contribution to $g_{n,k}z^nq^k$ from words starting with $b$, and summing 
over all such words gives their contribution to $G(z;q)$.

We claim that 
\begin{align*}
\sum_{n,k}\left(\sum_{i=0}^{n-2}\sum_j \left(l_{i,jN}\right)\left( g_{n-i-2, 
k-jM}\right)\right)z^{n-2}q^k=\left[\Phi_{N,M} \circ L(z;q) \right]G(z;q)
\end{align*}
Let us start by computing the action of the $\Phi$ operator
\begin{align*}
\Phi_{N,M} \circ L(z;q) 
  &= \Phi_{N,M} \circ \left( \sum_r z^r \sum_j l_{r,i} q^{i} \right) \\
  &= \sum_r z^r \cdot \Phi_{N,M} \circ \left(  \sum_i l_{r,i} q^{i} \right) 
  &= \sum_r z^r \cdot \Phi_{N,M} \circ \left(  \sum_j \sum_d l_{r,Nj+d} 
q^{Nj+d} \right) \\
  &= \sum_r z^r \sum_j l_{r,Mj} q^{Nj} 
\end{align*}
Now let us expand the right hand side.
\begin{align*}
\left[\Phi_{N,M} \circ L(z;q) \right]  G(z;q)
&= \left( \sum_r z^r \sum_j l_{r,jN} q^{jM} \right) 
 \cdot \left( \sum_i z^i \sum_e g_{i,e}q^{e} \right)\\
&= \sum_{r,i} z^{r+i} \left( \sum_{j} q^{jM} \cdot l_{r,jN} \right)  \cdot 
\left( \sum_{e} q^e \cdot g_{i,e} \right) \\
&= \sum_{r,i} z^{r+i} \sum_{j,e} q^{jM+e} \cdot l_{r,jN} \cdot g_{i,e}
\intertext{Set $r=\alpha, i=n-2-\alpha$ and $e=k-jN$}
&= \sum_{n,\alpha} z^{n-2} \sum_{j,k} q^{k} \cdot l_{\alpha,jN} \cdot 
g_{n-2-\alpha,k-jN}
% = & \sum_j \left( l_{0,jN}q^{jM} +z l_{1,jN}q^{jM} +z^2 l_{2,jN}q^{jM} 
% +\dots\right) \sum_e\left(  g_{0,e}q^{e} +z g_{1,e}q^{e} +z^2 g_{2,e}q^{e} 
% +\dots \right)\\
% \\
% = & \sum_j   \sum_e  \left(
% l_{0,jN}g_{0,e}+z(l_{0,jN}g_{1,e}+l_{1,jN}g_{0,e})+z^2(l_{0,jN}g_{2,e}+l_{1,jN}g_{1,e}+l_{2,jN}g_{0,e})+\dots
% \right)q^{e+jM}\\
% \\
% = & 
%  \sum_j \sum_e \left(\sum_{\alpha=0}^{\infty} \sum_{i=0}^{\alpha}  l_{i,jN} g_{\alpha-i,e} z^{\alpha}  \right)q^{e+jM}\\
% \\
% = & 
%  \sum_j\sum_{k-jM}  \sum_{n-2=0}^{\infty} \sum_{i=0}^{n-2}\left(  l_{i,jN}\right)\left( g_{n-2-i,k-jM}\right) z^{n-2}  q^{k}\\
% \\
% = & \sum_{n,k}\left(\sum_{i=0}^{n-2}\sum_j \left(l_{i,jN}\right)\left( g_{n-i-2, k-jM}\right)\right)z^{n-2}q^k
% 
\end{align*}
which is the required form.

%%(note we made substitutions $\alpha=n-2, k=e+jM$ in the steps above) ALSO moving some sums past each other slightly dodgy.

\item The first letter is $\bbar$. Factor $w \equiv u v$ where $u$ is the
shortest word so that $u \in \suba{}$. As per the previous case, $u \equiv
\bbar u' b$ for some $u' \in \suba{M}$ with $u' \in \mathcal{K}$, and  so $u \in \suba{N}$. 
It follows (by similar reasoning) that the contribution of these words to 
$G(z;q)$ is $z^2 \left( \Phi_{M,N} \circ K(z;q)\right) \cdot G(z,q)$.
\end{itemize}

We now prove the equation satisfied by $L(z,q)$:
\begin{align*}
 L &=  1 + z(q+\qbar)L 
  + z^2 L\cdot \left[  \Phi_{N,M}\circ L \right]
  + z^2 \left(L - \left[ \Phi_{N,N} \circ L \right] \right)
  \cdot \left[ \Phi_{M,N} \circ K \right] 
\end{align*}
Consider an element $w \in \mathcal{L}$, and we note that $\mathcal{L}$
(and $\mathcal{K}$) is closed under appending the generators $a$ and
$\abar$, but not prepending. For this reason we will factor words in $\mathcal{L}$ recursively by considering
the last letter of $w$, which again gives use five cases.
%See Figure~\ref{fig:decompL}.
% In a similar manner
%to the above, we factor words in $\mathcal{L}$ recursively by considering
%the last letter of $w$.
\begin{itemize}
\item $w$ is the empty word, accounting for the 1 in the expression.
\item The last letter is $a$ or $\abar$. Then $w \equiv v a$ or $w \equiv v 
\abar$ for some $v \in \mathcal{L}$, increasing the length by $1$ and altering 
the $a$-exponent by $\pm1$. This yields the term $z(q+\qbar) L(z;q)$.

\item The last letter is $\bbar$. Factor $w \equiv u v$ where $u$ is the
longest subword such that $u \in \suba{}$ and $v$ is non-empty. This forces $v 
\equiv b v' \bbar$ with the restriction that $v' \in \mathcal{L}$. Since both 
$v,v' \in \mathcal{L}$ we must have $v' \in \suba{N}$ and $v \in \suba{M}$. By 
similar arguments to those used above for $G$, the contribution of all such 
words is $z^2 L(z;q) \cdot \Phi_{N,M} \circ L(z;q)$.

\item The last letter is $b$. Factor $w \equiv u v$ where $u$ is the longest
subword such that $u \in \suba{}$ and $v$ is non-empty. This forces $v \equiv 
\bbar{} v' b$ with the restriction that $v' \in \mathcal{K}$. Lemma~\ref{lem 
Lchar} implies the subword $u\not\in \suba{N}$.

The generating function for $\left\{u \in \mathcal{L} \mid u \not\in
\suba{N}\right\}$ is given by $(L - \Phi_{N,N}\circ L)$, and so this last case
gives  $z^2 (L(z;q) - \Phi_{N,N} \circ L(z;q) ) \cdot \Phi_{M,N} \circ K(z;q)$. 
Again the details are similar to those used in the above argument for $G$.
\end{itemize}
Putting all of these cases together and rearranging gives the result.
The equation for $\mathcal{K}$ follows a similar argument.
\end{proof}
%\begin{figure}[h]
% \centering
%\includegraphics[height=3cm]{D2.pdf}
%\caption{Any word in $\mathcal{L}$ can be decomposed by considering its last
%letter. This results in the four possible factorisations we have drawn
%here. The subwords $L, L' \in\mathcal{L}$, $K\in\mathcal{K}$ and $u$ is a word
%in $\mathcal{L}$  in the subgroup $\langle a \rangle$,
%but not in the subgroup $\langle a^N \rangle$.}
%\label{fig:decompL}
%\end{figure}

\section{Solution for $\mathrm{BS}(N,N)$}
\label{sec:NN}

Let $G(z;q)$ be the function defined above for the group $\mathrm{BS}(N,N)$. Recall that our main interest is the coefficient of $q^0$ in $G(z;q)$ rather than the full generating function itself, so
let $[q^0]G(z;q)=\sum_{n} g_{n,0} z^n$ which by Lemma \ref{lem:zeroterms} is the cogrowth series for $\mathrm{BS}(N,N)$.
%

%1-z^2 L_0(z;q)^2 rather than L_0^1
%denominator in same sum should have (w^j *q + 1/w^jq) rather than (wq+1/wq)

\begin{theorem}\label{thm:main}
The generating function $G(z,q)$ is algebraic. Consequently the cogrowth series 
$[q^0]G(z;q)$ is D-finite, and the cogrowth rate is an algebraic number. 
\end{theorem}

\begin{proof}
We claim that the generating function $G(z;q)$ is algebraic, satisfying a 
polynomial equation (of degree $N+1$). We prove this claim below and first 
show that the remainder of the theorem follows from this claim.

Assume $G(z;q)$ is algebraic; since all algebraic functions are D-finite 
\cite{\Stanley}, it is also D-finite. The series $G(zq;q^{-1})$ is also 
algebraic and D-finite. The cogrowth series can then be expressed as the 
diagonal of this series, where the diagonal of a power series in $z$ and $q$ is 
defined to be
\begin{align*}
  \diag_{z,q} \circ \sum_{n,k} f_{n,k} z^n q^k &= \sum_n f_{n,n} z^n.
\end{align*}
In particular
\begin{align*}
  \diag_{z,q} \circ G(zq;q^{-1}) &= \diag_{z,q} \circ \sum_{n,k} z^n q^{n-k} 
g_{n,k} 
% \\  &= \diag_{z,q} \circ \sum_{n,i} z^n q^i g_{n,n-i} 
= \sum_n z^n g_{n,0}. 
\end{align*}
A result of Lipshitz~\cite{\LipsDiag} states that the diagonal of a 
D-finite series is itself D-finite. Thus $\diag_{z,q} G(zq;q^{-1}) = [q^0] 
G(z;q)$ is D-finite.

By definition, any D-finite generating function, $f(z)$, satisfies a linear differential equation 
with polynomial coefficients which can be written as
\begin{align*}
  f^{(n)}(z) + \sum_{j=0}^{n-1} \frac{p_j(z)}{p_n(z)} f^{(j)}(z) &= 0
\end{align*}
where the $p_i(z)$ are polynomial.

The singularities of $f(z)$  correspond to the 
singularities of the coefficients of the DE (see Chapter~1 of \cite{\Poole}).
Thus the singularities of the 
solution and its radius of convergence are all algebraic numbers. The 
exponential growth rate of the coefficients of the expansion of $f(z)$ about 
zero is equal to the reciprocal of the radius of convergence (see, for example, 
Theorem IV.7 in \cite{\Flajolet}). Thus the cogrowth rate is an algebraic 
number.

So to establish the theorem we need to show that $G(z;q)$ satisfies an 
polynomial equation.

Since $N=M$ we have  $K(z;q) = L(z;q)$ by Lemma  \ref{lem:LequalsK}, and so the 
equations in Proposition~\ref{prop:GLKeqns} simplify considerably to
 \begin{align*}
 L &=  1 + z(q+\qbar)L + 2 z^2 L\cdot
  \left[  \Phi_{N,N}\circ L \right] 
   - z^2 \left[ \Phi_{N,N} \circ L \right]^2, \\
% \intertext{and}
  %%
  G &= 1 + z(q+\qbar)G 
  + 2 z^2 G \cdot \left[ \Phi_{N,N} \circ L \right].
\end{align*}
To simplify notation in what follows, write
\begin{align*}
L_0(z;q) = \Phi_{N,N} \circ L(z;q)=\sum_{n,d} l_{n,dN}z^nq^{dN}.
\end{align*}
So the equations for $L$ and $G$ may be written as 
\begin{align*}
  L(z;q) &= \frac{1 - z^2 L_0(z;q)^2}{1 - z(q+\qbar) -2z^2 L_0(z;q)}, \\
  G(z;q) &= \frac{1}{1-z(q+\qbar)-2z^2 L_0(z;q)}.
\end{align*}
To complete the proof it suffices to show that $L_0(z;q)$ satisfies an 
algebraic equation of degree $(N+1)$.

Let $\omega = e^{2\pi i/N}$ be an $N^{th}$ root of unity. We claim that
\begin{align*}
    L_0(z;\omega^j q) &= L_0(z;q) &\text{and} \\
    N L_0(z;q) &= \sum_{j=0}^{N-1} L(z; \omega^j q). 
\end{align*}
%\textbf{old: The first of these follows from the fact that the coefficients of $z$ 
%in $L_0(z;q)$ are Laurent polynomials in $q^N$. }
To see the first, by definition  $L_0(z;\omega^j q) =\sum_{n,d} l_{n,dN}z^n(\omega^j q)^{dN}$ and $\omega^N=1$.
To derive the second consider 
the sum $\sum_{j=0}^{N-1} \omega^{kj}$, with $k$ in $\mathbb{Z}$. If $k = dN, 
d\in \mathbb{Z}$ then
\begin{align*}
  \sum_{j=0}^{N-1} \omega^{k} &= 
  \sum_{j=0}^{N-1} \omega^{Nd}   =\sum_{j=0}^{N-1}1 = N.
\end{align*}
On the other hand, if $k = Nd+l$ with $0<l<N$:
\begin{align*}
  \sum_{j=0}^{N-1} \omega^{kj} &= 
  \sum_{j=0}^{N-1} \left({\omega^N}\right)^{dj}\cdot \omega^{jl} = 
\sum_{j=0}^{N-1} \omega^{jl} \\
  & = 1 + \omega^l + \dots +   \omega^{(N-1) l} = \frac{1-\omega^{Nl} 
}{1-\omega^l} = 0
\end{align*}
since $\omega^l \neq 1$ and $\omega^{Nl}=1$. Returning to the second claim
\begin{align*}
  \sum_{j=0}^{N-1} L(z; \omega^j q) 
  &= \sum_{j=0}^{N-1} \sum_{n,k} g_{n,k} z^n q^k \omega^{kj} \\
  &= \sum_{n,k} g_{n,k} z^n q^k \sum_{j=0}^{N-1} \omega^{kj} \\
  &= \sum_{n,d} g_{n,dN} z^n q^{dN} \cdot N = NL_0(z;q)
\end{align*}
as required.

Combining the above expression for $L(z;q)$ in terms of $L_0(z;q)$ with that 
expressing $L_0(z;q)$ in terms of $L(z;q\omega^j)$ we obtain
\begin{align}
  \label{eqn nl0}
    N L_0(z;q) &= \sum_{j=0}^{N-1}
\frac{1 - z^2 L_0(z;q)^2}{1 - z(q\omega^j+\qbar\omega^{-j}) -2z^2  L_0(z;q)} 
\end{align}
where we have used the fact that $L_0(z;q\omega^j) = L_0(z;q)$. Clearing the 
denominator of this expression we have
\begin{multline*}
  N L_0(z;q) \cdot \prod_{j=0}^{N-1} 
\left( 1 - z(q\omega^j+\qbar\omega^{-j})-2z^2  L_0(z;q) \right)
\\= \sum_{j=0}^{N-1} \left( 1 - z^2 L_0(z;q)^2 \right) 
\cdot \prod_{\substack{0\leq l\leq N \\ l\neq j}}
\left( 1 - z(q\omega^j+\qbar\omega^{-j})-2z^2  L_0(z;q) \right).
\end{multline*}
Thus $L_0(z;q)$ satsifies an algebraic equation of degree at most $N+1$.

%Rewriting $L_0(z;q)$ in terms of either $L(z;q)$ or $G(z;q)$ and substituting 
%it back into the above expression, will result in an algebraic equation for $L$ 
%and $G$ (respectively) of degree at most $N+1$.
%
%Now rewrite $L_0(z;q)$ in terms of $G(z;q)$:
%\begin{align*}
%  L_0(z;q) &= \frac{1}{2z^2} - 
%\frac{q+\qbar}{2z}- \frac{1}{2z^2 G(z;q)} 
%\end{align*}
%Substituting this into the above equation for $L_0(z;q)$ and clearing 
%denominators will also result in an equation for $G(z;q)$ of degree at most 
%$N+1$.

Now rewrite $L_0(z;q)$ in terms of $G(z;q)$:
\begin{align}
  \label{eqn l0g}
  L_0(z;q) &= \frac{1}{2z^2} - 
\frac{q+\qbar}{2z}- \frac{1}{2z^2 G(z;q)}. 
\end{align}
Substituting this into the above equation for $L_0(z;q)$ and clearing 
denominators yields an equation for $G(z;q)$ of degree at most 
$N+1$.
 \end{proof}

We remark that for small values of $N$, the equation satisfied for $G$ is 
relatively easy to write down (with the aid of a computer algebra system).
Set $Q = q + \qbar$, then for $N=2,3,4,5$  the generating function $G(z;q)$ 
satisfies the following $(N+1)$-degree polynomial equations
\begin{align}
\label{eqn bs22 gzq}
1+3zQG - (1-4z^2-z^2Q^2)G^2 - zQ(1-zQ-2z)(1-zQ+2z)G^3 & = 0,
\end{align}
%For $\mathrm{BS}(3,3)$ the generating function $G(z;q)$ satisfies the
% quartic equation
\begin{multline}
1
+4zQG
+(6Q^2z^2-z^2-1)G^2
+2z(Qz+1)(Q^2z-Q+2z)G^3\\
+z^2(1-Q)(1+Q)(Qz+2z-1)(Qz-2z-1)G^4 = 0,
\end{multline}
%For $\mathrm{BS}(4,4)$ the generating function $G(z;q)$ satisfies the
% quintic equation
\begin{multline}
1
+5GQz
+(10Q^2z^2-2z^2-1)G^2
+z(10Q^3z^2-6Qz^2-3Q+4z)G^3\\
+z^2(3Q^4z^2+2Q^2z^2-3Q^2+8Qz-8z^2+2)G^4\\
-z^3Q(Q^2-2)(Qz+2z-1)(Qz-2z-1)G^5 =0
\end{multline}
and %\adr{(though things are starting to get ugly here)}
\begin{multline}
1 +6 Qz G  +(15Q^2z^2-3z^2-1)G^2 \\
+4(5Q^3z^2-3Qz^2-Q+z)z G^3 
+3(5Q^4z^2-6 Q^2z^2-2Q^2+4Qz-z^2+1)z^2 G^4\\
+2(2Q^5z^2-Q^3z^2-2Q^3+6Q^2z-8Qz^2+3Q-4z)z^3 G^5\\
-(Q^2+Q-1)(Qz+2z-1)(Qz-2z-1)(Q^2-Q-1)z^4 G^6 = 0
\end{multline}

respectively.

\section{Cogrowth rates}\label{sec:rates}
Theorem \ref{thm:main} proves that the cogrowth rates are algebraic numbers 
but does not give an explicit method for computing them. Since for $N\geq 2$ the groups 
$\mathrm{BS}(N,N)$ are non-amenable, we know by Grigorchuk and 
Cohen \cite{\Grig,\Cohen} these numbers are strictly less than $4$. The 
following theorem enables us to compute the cogrowth rates explicitly.

Recall that the reciprocal of the radius of convergence of a power series 
$\sum_n a_nz^n$ is $\limsup a_n^{1/n}$.
% , which is the rate of exponential growth of the sequence $\{a_n\}$.
\begin{theorem}
\label{thm gz1 gz0}
For $\mathrm{BS}(N,N)$, the generating functions $G(z;1)$ and $[q^0]G(z;q)$ have 
the same radius of convergence. Hence the cogrowth rate is the reciprocal of 
the radius of convergence of $G(z,1)$.
\end{theorem}
The proof of the above theorem hinges on the following observation
\begin{lemma}
  \label{lem gnk 0}
 For $\mathrm{BS}(N,N)$, $g_{n,k}=0$ for $|k|>n$.
\end{lemma}
\begin{proof}
Let $w$ be a word of length $n$ and $a$-exponent $k$ in $\mathcal H$. If $w$ contains a canceling pair of generators then removing them does not increase length, and replacing $a^{\pm N}b^{\pm 1}$ with $b^{\pm 1}a^{\pm N}$ similarly does not increase length, so applying these moves to $w$ we obtain a freely reduced word of length at most $n$ of the form $a^{\eta_1}b^{\epsilon_1}\dots a^{\eta_s}b^{\epsilon_s}a^r$ where $\eta_i,\epsilon_i,r,s\in\mathbb Z, |\eta_i|<N, \epsilon_i=\pm 1$ and $s\geq 0$. Since $w\in H$ we have $w=a^p$ for some $p\in \mathbb Z$, so 
$a^{\eta_1}b^{\epsilon_1}\dots a^{\eta_s}b^{\epsilon_s}a^{r}a^{-p}=1$, and Britton's lemma implies that $s=0$ and $p=k\leq n$.
\end{proof}

\begin{proof}[Proof of Theorem~\ref{thm gz1 gz0}]
Let $g_n$ denote the number of words in $\mathcal H$ of length $n$, that is, 
$g_n=\sum_k g_{n,k}$, and $G(z;1) = \sum_{n=0}^\infty g_n z^n$. Write $\limsup 
g_n^{1/n} = \mu_{all}$ and $\limsup g_{n,0}^{1/n} = \mu_0$. Since $g_{n,k} \geq 0$ we 
have that $g_n \geq g_{n,0}$ and so $\mu_{all} \geq \mu_0$.

To prove the reverse inequality we proceed as follows. For a fixed $n$, by 
Lemma~\ref{lem gnk 0}, $g_{n,k}$ is positive for at most $2n+1$ values of $k$. 
Since $g_{n,k}$ is positive for only finitely many $k$, there is some integer 
$k^*$ (depending on $n$) so that $g_{n,k^*} \geq g_{n,k}$ for all $k$. Then we 
have 
\begin{align}
  \label{eqn gnkstar}
  g_{n,k^*} \leq g_n \leq (2n+1) g_{n,k^*}
\end{align}
and hence $\limsup g_{n,k^*}^{1/n} = \mu_{all}$. 

Keeping $n$ fixed, consider a word that contributes to $g_{n,k^*}$ and another
that contributes to $g_{n,-k^*}$. Concatenating them together gives a word that
contributes to $g_{2n,0}$. 
%Then considering all possible words that are obtained 
%by the concatenation of $M$ words of length $n$ and $a$-exponent $k^*$ with $M$ 
%words of length $n$ and $a$-exponent $-k^*$ we obtain the following inequality
It follows that \begin{align}
  g_{n,k^*}g_{n,-k^*}  & \leq g_{2n,0}.
\end{align}
Since 
 $g_{n,k^*}=g_{n,-k^*}$ by 
Lemma~\ref{lem:gnkneg} this gives
\begin{align}
  g_{n,k^*}^2 & \leq g_{2n,0}
\end{align}

 Raising both sides to the power $\nicefrac{1}{2n}$ and taking $\limsup_{n\rightarrow \infty}$ gives
\begin{align}
 \limsup_{n\rightarrow \infty} g_{n,k^*}^{1/n} &\leq \limsup_{n\rightarrow \infty} \left(g_{2n,0}\right)^{1/2n}
\end{align}
and so  $\mu_{all} \leq \mu_0$.
\end{proof}

The preceding proof is known as a ``most popular'' argument  in statistical 
mechanics, where it is commonly used to prove equalities of free-energies
(see \cite{\Hammersley} for example). In our case a most popular value is 
$k^*$, since this is a value that maximises $g_{n,k}$.
Numerical experiments show that $k^* = 0$, with a distribution tightly peaked 
around 0. This tells us experimentally that the most numerous words evaluating
to an element in $\suba{}$ are those with zero $a$-exponent.

For fixed $N$, we can use a computer algebra system to combine 
equations~\Ref{eqn nl0} and~\Ref{eqn l0g} in the proof of 
Theorem~\ref{thm:main} to obtain explicit polynomial equations satisfied by 
$G(z,q)$. We listed these for $N=2,3,4,5$ at the end of the previous section. 
Substituting $q=1$ (and $Q=q+\qbar=2$) gives an explicit polynomial equation 
for $G(z,1)$. One can then find its radius of convergence by solving the 
discriminant of the polynomial --- see for example \cite{\KW}. Alternatively one 
can convert the algebraic equation satisfied by $G(z,1)$ to a differential 
equation and then a linear recurrence satisfied by $g_{n}$ (using say, the 
\texttt{gfun} package \cite{gfun} for the \texttt{Maple} computer algebra 
system). The asymptotics of $g_n$ can then be determined using the methods 
described in \cite{\ZW}.

%In Maple the commands are ......
%so - in maple - you just say something like "discrim(poly, G)";
%it compute the discrim
%which will be a poly in z
%fsolve(blah,z) will then find the solutions for z
%the smallest is z_c
%= radius of conv

 These computations become extremely slow for even 
modest values of $N$. We give the results for $1\leq N \leq 10$ 
in Table~\ref{tab_mu_NN}.

Some simple numerical analysis of these numbers suggests that the
cogrowth rate approaches $\sqrt{12}$ exponentially with increasing $N$. This
finding agrees with work of Guyot and Stalder \cite{\GStald}, 
who examined the limit of the {\em marked} groups $\mathrm{BS}(M,N)$ as $M,N\rightarrow
\infty$, and found that the groups tend towards the free group on two letters,
which has an asymptotic cogrowth rate of  $\sqrt{12}$.

\begin{table}
\begin{center}
\begin{tabular}{|c|c||c|}
\hline
$N$ & $\mu$ & $\lambda$ \\
\hline
1& 4 & 3 \\
2& 3.792765039 & 2.668565568 \\
3& 3.647639445 & 2.395062561 \\
4& 3.569497357 & 2.215245886 \\
5& 3.525816111 & 2.091305394 \\
6 & 3.500607636 & 2.002421757\\
7 &  3.485775158 & 1.936941986 \\
8 & 3.476962757 & 1.887871818 \\
9 & 3.471710431 & 1.850717434 \\
10 & 3.468586539 & 1.822458708 \\
\hline
\end{tabular}
\end{center}
\caption{The cogrowth rate $\mu$ in $\mathrm{BS}(N,N)$ and the
corresponding growth rate of freely reduced trivial words $\lambda$. Note
that $\mu$ and $\lambda$ are related by $\mu = \lambda + 3/\lambda$ (see 
Lemma~\ref{lemma woe}), and that the cogrowth rate in the
free group on 2 generators is $\sqrt{12} = 3.464101615$.}
\label{tab_mu_NN}
\end{table}

\begin{remark}
  For $\mathrm{BS}(N,M)$ with $N \neq M$, $g_{n,k}$ can be nonzero for $|k|>n$ (for example in $BS(2,4)$ 
we have $t^ia^2t^{-i}=a^{2^{i+1}}$).
So while the left-hand inequality in equation~\Ref{eqn 
gnkstar} still holds, the factor of $(2n+1)$ in the right-hand inequality would 
be replaced with a term that grows exponentially with $n$, which means the   
proof of Theorem~\ref{thm gz1 gz0}
breaks down.
 We computed $g_{n,k}$ for small $n$ for several values of $N,M$ 
and while the number of $k$-values for which $g_{n,k}>0$ grows 
exponentially with $n$,
we observed that the 
distribution of $g_{n,k}$ is again very tightly peaked around $k=0$. This suggests 
that $g_{n,0}$ is the dominant contribution to $g_n$ and thus it may still be 
the case that $G(z,1)$ and $[q^0]G(z,q)$ have the same radii of convergence.
\end{remark}

\section{Discussion of differential equations satisfied by cogrowth series}
\label{sec:numerical}
Our proof of Theorem~\ref{thm:main} guarantees the existence 
of differential equations satisfied by the cogrowth series. It does, however, 
give a recipe for producing them. In this section we use a recently developed 
algorithm due to Chen, Kauers and Singer \cite{\Kauers} to obtain explicit 
differential equations for small values of $N$.

Applying the algorithms described in \cite{\Kauers} to the generating function
$G(z;q)$ for $\mathrm{BS}(2,2)$ we found a $6^\mathrm{th}$ order linear 
differential equation satisfied by $[q^0]G(z;q)$. The polynomial coefficients 
of this equation have degrees up to 47 and so we have not listed it 
here\footnote{The differential equation for the cogrowth of $\mathrm{BS}(N,N)$ 
for $1 \leq N \leq 10$ can be found at 
\texttt{http://www.math.ubc.ca/$\sim$andrewr/pub$\_$list.html}}. 
Again applying the same methods, we found an ODE of order 8 with coefficients of 
degree up to 105 for $\mathrm{BS}(3,3)$ and for $\mathrm{BS}(4,4)$ it is order 
10 with coefficients of degree up to 154.  We thank Manuel Kauers for his 
assistance with these computations.

By studying these differential equations we can determine the asymptotic 
behaviour of $g_{n,0}$ in more detail and demonstrate that the cogrowth series 
is not algebraic.
\begin{prop}
  The coefficients of the cogrowth series grow as
  \begin{align*}
    g_{n,0} &= \binom{n}{n/2}^2 \sim \frac{2}{\pi n} \cdot 4^n & N=1 \\
    g_{n,0} &\sim \frac{A_N}{n^2} \cdot \mu^n & N=2,3,4
  \end{align*}
  for even $n$ where $A_N$ is some real number. As a consequence the 
cogrowth series for $\mathrm{BS}(N,N)$ is not algebraic for $N=1,2,3,4$.
\end{prop}
\begin{proof}
We can compute the differential equations satisfied by the cogrowth series 
using the techniques in \cite{\Kauers}. From these equations we then use the 
 methods 
developed in \cite{\ZW} to determine the asymptotics of $g_{n,0}$. The presence 
of the factors of $n^{-1}$ and $n^{-2}$ in the asymptotic forms proves (see 
Theorem~D from \cite{flajolet1987}) that the associated generating function 
cannot be algebraic.
\end{proof}

While there is no theoretical barrier at $N=4$, the derivation of differential 
equations by this method slows quickly with increasing $N$ and we were unable 
to compute the differential equations for $N \geq 5$ rigorously. However, using 
clever guessing techniques
(see \cite{\KauersGuess} for a description) Manuel Kauers also found 
differential equations for $N=5,\dots,10$. For $\mathrm{BS}(5,5)$ the DE is 
order 12 with coefficients of degree up to 301, while that of 
$\mathrm{BS}(10,10)$  is 22nd order with coefficients of degree up to 1153 --- the computation for $N=10$ took about 50 days of computer time,
and when written in text
file is over 6 Mb! 

Using these differential equations we also determine that for $N\leq 10$ we have
\begin{align*}
  g_{n,0} \sim \frac{A_N}{n^2} \mu_N^n.
\end{align*}
In light of this evidence we make the following conjecture:
\begin{conj}
The number of trivial words in $\mathrm{BS}(N,N)$ grows as
\begin{align*}
  g_{n,0} &\sim \frac{A_N}{n^2} \cdot \mu_N^n & \text{for even $n$ }
\end{align*}
for $N \geq 2$ and consequently the cogrowth series is not algebraic.
\end{conj}

\section*{Acknowledgements}
The authors thank Manuel Kauers for assistance with establishing the
differential equations described in Section~\ref{sec:numerical}. Similarly we thank
Tony Guttmann for discussions on the analysis of series data. Finally we would
like to thank Sean Cleary and Stu Whittington for many fruitful discussions.

This research was supported by the Australian Research Council (ARC), the
the
Natural Sciences and Engineering Research Council of Canada (NSERC), and the Perimeter Institute for Theoretical
Physics.  Research at Perimeter Institute is supported by the Government
of Canada through Industry Canada and by the Province of Ontario through
the Ministry of Economic Development and Innovation.

\bibliographystyle{plain}
\bibliography{refs-cogrowth}

\end{document}